\documentclass[12pt]{article}
\usepackage{verbatim}
\usepackage{latexsym}
\usepackage{amsmath,amsthm,amsfonts}
\usepackage{eucal}
\usepackage{cases}

\usepackage{mathtext}
\usepackage[cp1251]{inputenc}
\usepackage[T2A]{fontenc}
\usepackage[dvips]{graphicx}
\usepackage{amsmath}
\usepackage{amssymb}
\usepackage{amsxtra}
\usepackage{latexsym}
\usepackage{ifthen}

\newtheorem{theorem}{Theorem}
\newtheorem{definition}{Definition}[section]
\newtheorem{lemma}{Lemma}[section]

\newtheorem{remark}{Remark}[section]

\numberwithin{equation}{section}

\begin{document}

\title {\bf New dissipated energy for the unstable thin film equation}

\author{{\normalsize\bf Marina Chugunova, Roman M. Taranets}
\smallskip}

\def\lhead{R.M. Taranets}

\def\rhead{Cahn-Hilliard}

\date{\today}

\maketitle

\setcounter{section}{0}

\begin{abstract}
The fluid thin film equation $h_t = - (h^n h_{xxx})_x - a_1\,(h^m
h_x)_x$ is known to conserve mass $\int\,h \, dx$, and in the case
of $a_1 \leq~0$, to dissipate entropy $\int\,h^{3/2 - n}\,dx$ (see
\cite{BertozziBrenner}) and the $L^2$-norm of the gradient
$\int\,h_x^2\,dx$ (see \cite{B8}). For the special case of $a_1 = 0$
a new dissipated quantity $\int\, h^{\alpha}\,h_x^2\,dx $ was
recently discovered for positive classical solutions by Laugesen
(see \cite{Laugesen}). We extend it in two ways. First, we prove
that Laugesen's functional dissipates strong nonnegative generalized
solutions. Second, we prove the full $\alpha$-energy
$\int\,\bigl(\tfrac{1}{2} \,h^\alpha \, h_x^2\, - $ $ \tfrac
{a_1\,h^{\alpha + m - n + 2}}{(\alpha + m - n + 1)(\alpha + m - n +
2)} \bigr)\, dx $ dissipation for strong nonnegative generalized
solutions in the case of the unstable porous media perturbation
$a_1> 0$ and the critical exponent $m = n+2$.
\end{abstract}


\textbf{keywords:} {fourth-order degenerate parabolic equations,
thin liquid films, energy, entropy }

\section{Introduction} \label{A}

It is well known that analysis of the existence, uniqueness and
regularity of weak solutions for nonlinear evolution equations
relies heavily on a priori estimates. Often, the physical energy or
entropy which originate from the related model can provide
non-increasing in time quantities. Unfortunately, it is far from
obvious how to construct new non-increasing Lyapunov type
functionals. A general algebraic approach to the construction of
entropies in higher-order nonlinear PDEs can be found in \cite{JM}
and can be applied to analyse thin film equations with stabilizing
porus media type perturbations. In this paper, inspired by
Laugesen's result \cite{Laugesen} on dissipation, we prove that the
energy functional introduced in \cite{Laugesen} dissipates strong
nonnegative generalized solutions. However, our method of the proof
is only applicable to some subset of the Laugsen's dissipation
region \cite{Laugesen} (see the shaded area on Figure~1).

We study the longwave-unstable generalized thin film equation
\begin{equation}
\label{A.gen.thin-film} h_t = - (h^n\,h_{xxx})_x -
a_1\,(h^m\,h_x)_x,
\end{equation}
where $h(x,t)$ gives the height of the evolving free-surface. The
exponent $n$ plays a stabilizing role due to fourth-order forward
diffusion term and the exponent $m$ plays a destabilizing role due
to backward second-order diffusion term for the case when $a_1 > 0$.
This class of equations originates from many physical/industrial
applications involving air-fluid interface. For example: the case
$n=1, \quad m=1$ describes a thin jet in a Hele-Shaw cell
\cite{Constantin}, the case $n = 3, \quad m = -1$ describes Van der
Waals driven rupture of thin films \cite{wit1}, the case $ m = n =
3$ describes shape of fluid droplets hanging from a ceiling
\cite{Enrhard}, and the case $n = 0, \quad m = 1$ describes
solidification of a hyper-cooled melt (this is a modified
Kuramoto-Sivashinsky equation) \cite{Bernoff}.

To prove that the nonnegativity property is preserved in nonlinear
thin film equation $h_t = -(h^n\,h_{xxx})_x$ for $n \geq 1$ (case
$a_1=0$) Bernis and Friedman \cite{B8} used set of dissipated and
conserved quantities: mass conservation $\int\,h\,dx = M$, surface
energy dissipation $\frac{d}{dt} \int {h_x^2\,dx} \leq 0$, and
entropy dissipation $\frac{d}{dt} \int{h^{2-n}\,dx} \leq 0$. The new
so-called $\beta$-entropy $\int\,h^{2-n +\beta}\,dx$ was introduced
by Bertozzi and Pugh \cite{BertPugh1996} and independently and
simultaneously by Beretta, Bertsch, Dal~Passo \cite{B2} to extend
this result to $n > 0$. They also successfully used this new entropy
to obtain exponential with respect to the $L^{\infty}$-norm
convergence toward the mean value steady state solution. To analyse
this convergence rate in $H^1$-norm for the special case $n=1,\  a_1
= 0$ Carlen and Ulusoy \cite{CarUl} used the dissipated energy $\int
{h^{\alpha}\,h_x^2\,dx} $ constructed by Laugesen \cite{Laugesen}
for classical positive solutions. Exponential asymptotic convergence
toward the mean value was also studied  by Tudorascu in
\cite{Tudorascu}. This list of connections between new properties of
solutions in thin film PDEs proved by means of newly discovered
dissipated quantities is far from complete.

In this paper we prove that there exists a subinterval $I$ of $ - 1
< \alpha <~1$ ($I$ depends on $n$ only) and a nonnegative strong
generalized solution such that for any $\alpha \in I$ the full
$\alpha$-energy
$$ \mathcal{E}_{0}^{(\alpha)}(t) =
\int\limits_{\Omega} {\left(\tfrac{1}{2} \,h^\alpha \, h_x^2\, -
\tfrac {a_1\,h^{\alpha + m - n + 2}}{(\alpha + m - n + 1)(\alpha +
m - n + 2)} \right)\, dx}
$$
dissipates. For the unstable porus media perturbation case $a_1 > 0$
this dissipation is proven under the assumptions that the total mass
of the solution is less than or equal to the critical one, $m = n+2$
and domain $\Omega$ is unbounded or $h$ is compactly supported. For
the stable case $a_1 \leq 0$ no such assumptions are needed.

We proceed as follows.  First, we show the dissipation for the
classical solutions of the regularized problem and then we take this
dissipation to the limit. We prove dissipation of the full
$\alpha$-energy for positive classical solutions of the regularized
problem for any value of the coefficient $a_1$ and without any
additional assumptions about the total mass of the solution or its
support. However our method of taking the dissipation to the limit
due to the Bernis-Friedman method of regularization requires
additional conditions for the case $a_1 > 0$.

\section{Auxiliary results to generalized weak solutions} \label{B}

We consider nonnegative weak  solutions to the following
initial--boundary  problem:
\begin{numcases}
{(\textup{P})}
 h_t  + \left(
{h^n h_{xxx} + a_1 h^m h_x } \right)_x = 0 \text{ in }Q_T, \hfill \label{B:1}\\
\tfrac{\partial^{i} h}{\partial x^i}(-a,t) = \tfrac{\partial^{i}
h}{\partial x^i}(a,t) \text{ for }
t > 0,\,i=\overline{0,3}, \hfill \label{B:2}\\
\qquad  \qquad h(0,x) = h_0 (x) \geqslant 0, \hfill \label{B:3}
\end{numcases}
where $h= h(t,x)$, $\Omega = (-a, a)$, $Q_T = (0,T) \times \Omega$,
$n > 0$, $m > 0$,  and $a_1 \in \mathbb{R}^1$. We define a
generalized weak solution in the Bernis-Friedman sense (see, e.\,g.
\cite{B2,B8}).

\begin{definition}[generalized weak solution] \label{B:defweak}
Let $n > 0$, $m > 0$, and $a_1 \in \mathbb{R}^1$. A generalized
weak solution of problem $($P$)$ is a function $h$ satisfying
\begin{align}
& \label{weak1}
h \in C^{1/2,1/8}_{x,t}(\overline{Q}_T) \cap L^\infty (0,T; H^1(\Omega )),\\
& \label{weak-d} h_t \in L^2(0,T; (H^1(\Omega))'),\\
& \label{weak2} h \in C^{4,1}_{x,t}(\mathcal{P}), \,\,\,
h^{\frac{n}{2}} ( h_{xxx} + a_1 h^{m-n} h_x ) \in
L^2(\mathcal{P}), \,\,
\end{align}
where $\mathcal{P} = \overline{Q}_T \setminus ( \{h=0\} \cup
\{t=0\})$ and $h$ satisfies (\ref{B:1}) in the following sense:
\begin{equation}\label{integral_form}
\int\limits_0^T \langle h_t(\cdot,t), \phi \rangle \; dt -
\iint\limits_{\mathcal{P}} {h^n ( h_{xxx} + a_1 h^{m-n}
h_x)\phi_x\,dx dt } \ = 0
\end{equation}
for all $\phi \in C^1(Q_T)$ with $\phi(-a,\cdot) = \phi(a,\cdot)$;
\begin{align}
&  \label{ID1} h(\cdot,t) \to h(\cdot,0) = h_0
\mbox{  pointwise \& strongly in $L^2(\Omega)$ as $t \to 0$}, \\
& \label{BC1} h(-a,t)=h(a,t) \; \forall t \in [0,T] \; \mbox{and}
\; \tfrac{\partial^{i}
h}{\partial x^i}(-a,t) = \tfrac{\partial^{i}h}{\partial x^i}(a,t)  \\
& \notag \mbox{for} \; i = \overline{1,3} \; \mbox{at all points
of the lateral boundary where $\{h \neq 0 \}$.}
\end{align}
\end{definition}

Because the second term of (\ref{integral_form}) has an integral
over $\mathcal{P}$ rather than over $Q_T$, the generalized weak
solution is ''weaker'' than a standard weak solution. Also note that
the first term of (\ref{integral_form}) uses $h_t \in L^2(0,T;
(H^1(\Omega))' )$; this is different from the definition of weak
solution first introduced by Bernis and Friedman \cite{B8}; there,
the first term was the integral of $h \phi_t$. The proof of the
existence of generalized weak solutions follows the ideas of
\cite{B8,B2,BertPugh1996,BP1,BP2,T4}.

Let
\begin{equation}\label{C:Galpha}
G_{0}^{(\beta)} (z): = \left\{ \begin{gathered} \tfrac{z^{\beta
-n+ 2 }}{(\beta - n + 2)(\beta - n + 1)} \text{
if } \beta - n \neq \{ - 1, - 2 \}, \\
z \ln z - z  \text{ if } \beta - n = - 1 , \\
- \ln z \text{ if } \beta - n = - 2,
\end{gathered} \right.
\end{equation}
$(G^{(\beta)}_{0} (z))'' = z^{\beta - n}$, and $G_{0} (z):=
G_{0}^{(0)} (z)$.

\begin{theorem}\label{C:Th1}
Let $a_1 \in \mathbb{R}^1$, $n > 0$; $m \geqslant n/2 $ for $a_1
> 0$, and $m > 0$ for $a_1 \leq 0$.

(a) [Existence.] Let the nonnegative initial data $h_0 \in
H^1(\Omega)$ satisfy
\begin{equation}\label{C:inval}
\int\limits_{\Omega} {G_0(h_0(x)) \,dx} < \infty,
\end{equation}
and either 1) $h_0(-a) = h_0(a) = 0$ or  2) $h_0(-a) = h_0(a) \neq
0$ and $ \tfrac{\partial^{i} h_0}{\partial x^i}(-a) =
\tfrac{\partial^{i}h_0}{\partial x^i}(a) \text{ holds for
}i=1,2,3$. Then for some time $T_{loc}>0$ there exists a
nonnegative generalized weak solution, $h$, on $Q_{T_{loc}}$ in
the sense of the definition \ref{B:defweak}.  Furthermore,
\begin{equation} \label{Linf_H2}
h \in L^2(0,T_{loc};H^2(\Omega)).
\end{equation}
Let
\begin{equation} \label{Energy}
\mathcal{E}_0(T) := \int\limits_{\Omega} { \{\tfrac{1}{2}
h_{x}^2(x,T) - a_1 D_0(h(x,T)) \}\,dx},
\end{equation}
where $D_0(z):= \frac{z^{m - n +2}}{(m - n + 1)(m- n + 2)}$. Then
the weak solution satisfies
\begin{equation}\label{C:d2'}
\mathcal{E}_0(T) + \iint\limits_{\{h >0 \}} {h^n (h_{xxx} + a_1
h^{m - n}h_x)^2 \,dx dt} \leqslant \mathcal{E}_0(0),
\end{equation}
\begin{equation}\label{C:ddd2'}
\iint\limits_{\{h >0 \}} {h^n h^2_{xxx}\,dx dt} \leqslant \,\,
const < \infty.
\end{equation}
for all $T \leq T_{loc}$. The time of existence, $T_{loc}$, is
determined by $a_1$, $| \Omega |$, $\int h_0$, $\| h_{0x} \|_2$,
and $\int G_0(h_0)$. Moreover, $T_{loc} = + \infty $ for $a_1 \leq
0$.

(b) [Regularity.] If the initial data from (a) also satisfies
$$
\int\limits_{\Omega} {G^{(\beta)}_{0}(h_0(x)) \,dx} <\infty
$$
for some $-1/2 < \beta < 1, \  \beta \neq 0$ then there exists $0 <
T_{loc}^{(\beta)} \leq T_{loc}$ such that the nonnegative
generalized weak solution has the extra regularity
\begin{equation}\label{Linf_H2-2}
h^{\tfrac{\beta + 2}{2}} \in L^{2}(0, T_{loc}^{(\beta)};
H^2(\Omega)) \text{ and } h^{\tfrac{\beta + 2}{4}} \in L^{2}(0,
T_{loc}^{(\beta)}; W^1_4(\Omega)).
\end{equation}
The time of existence, $T_{loc}^{(\beta)}$, is determined by
$a_1$, $| \Omega |$, $\int h_0$, $\| h_{0x} \|_2$, and $\int
G^{(\beta)}_{0}(h_0)$. Moreover, $T_{loc}^{(\beta)} = + \infty $
for $a_1 \leq 0$.
\end{theorem}

There is nothing special about the time $T_{loc}$ in the
Theorem~\ref{C:Th1}. In the case $a_1 >0$ and $n/2 \leq m < n +2$
(or $m = n+2$ and $M \leq M_c$), given a countable collection of
times in $[0,T_{loc}]$, one can construct a weak solution for which
these bounds will hold at those times. Also, we note that the
analogue of Theorem 4.2 in \cite{B8} also holds: there exists a
nonnegative weak solution with the integral representation
\begin{align} \label{alt_int}
& \int\limits_0^T \langle h_t(\cdot,t), \phi \rangle \; dt
+ \iint\limits_{Q_T} (n h^{n-1} h_x h_{xx} \phi_x + h^n h_{xx} \phi_{xx}) \; dx dt \\
& \hspace{1.5in} - a_1 \iint\limits_{Q_T} { h^m h_x \phi_x \; dx
dt} = 0. \notag
\end{align}

\section{Dissipation of energy for nonnegative weak solutions} \label{D}

The main result of the present paper is the following

\begin{theorem}\label{D:Th1}
Let $a_1 \in \mathbb{R}^1$, $1/2 < n < 3$; $m \geqslant n/2 $ for
$a_1 > 0$, and $m > 0$ for $a_1 \leq 0$, and
\begin{equation} \label{D:Energy-new0}
\mathcal{E}_{0}^{(\alpha)}(T) := \int\limits_{\Omega} {
\{\tfrac{1}{2} h^{\alpha} h_{x}^2(x,T) - a_1 \tilde{D}_{0}(h(x,T))
\}\,dx},
\end{equation}
where $\tilde{D}_{0}(z):= \frac{z^{\alpha +m -n+2}}{(\alpha +m
-n+1)(\alpha +m -n+2)}$, and $\mathcal{E}_{0}^{(0)}(T) =
\mathcal{E}_{0}(T)$. Then there exists a non-empty subinterval $I$
(see \cite{Laugesen} for the explicit form of the $I$) of $0 \leq
\alpha < 1$ for $\frac{1}{2} < n < 3$, and of $ \frac{3}{2} - n <
\alpha < 0$ for $\frac{3}{2} < n < 3$ such that for any $\alpha \in
I$ the nonnegative weak solution from Theorem~\ref{C:Th1} satisfies
the following estimates:

(i) if $a_1 \leqslant 0$ then
\begin{equation}\label{D:0}
\mathcal{E}^{(\alpha)}_{0}(T) \leqslant
\mathcal{E}^{(\alpha)}_{0}(0);
\end{equation}

(ii) if $a_1 > 0$ then
\begin{equation}\label{D:1}
\mathcal{E}^{(\alpha)}_{0}(T) \leqslant
\mathcal{E}^{(\alpha)}_{0}(0) + C_1 \iint \limits_{Q_T} {h^{\alpha
+ 3m - 2n + 2} dx dt}  \text{ for } m > n + 2;
\end{equation}
\begin{equation}\label{D:1-2}
\mathcal{E}^{(\alpha)}_{0}(T) \leqslant
\mathcal{E}^{(\alpha)}_{0}(0) +  T(C_1 M^{\frac{2\alpha + 5 m - 3n
+ 4}{n + 2 - m}} + C_2 M^{\alpha + 3m - 2n + 2})
\end{equation}
for $m < n + 2$ and $\alpha > 2n - 3m - 1$;
\begin{equation}\label{D:1-4}
\mathcal{E}^{(\alpha)}_{0}(T) \leqslant
\mathcal{E}^{(\alpha)}_{0}(0) +  C_3 T\, M^{\alpha + 3m - 2n + 2}
\end{equation}
for $m < n + 2$ and $2n - 3m - 2 < \alpha \leqslant 2n - 3m - 1$;
\begin{equation}\label{D:1-3}
\mathcal{E}^{(\alpha)}_{0}(T) \leqslant
\mathcal{E}^{(\alpha)}_{0}(0) +  C_2 T\, M^{\alpha + n + 8} \text{
for } m = n + 2 \text{ and } 0 < M \leqslant M_c.
\end{equation}
Here $C_2 = 0$ if $\Omega$ is unbounded or $h$ has compact support.
\end{theorem}

\begin{remark}[Extra Regularity]
In particular, the extra regularity $h^{\frac{\alpha + 2}{2}} \in
L^{\infty}(0,T; H^1(\Omega))$ follows directly from
Theorem~\ref{D:Th1}. Hence, $h^{\frac{\alpha + 2}{2}}(\cdot,T) \in
H^1(\Omega)$ for almost all $T \in [0,T_{loc}^{(\beta)}]$ and
therefore $h^{\frac{\alpha + 2}{2}}(\cdot,T) \in
C^{1/2}(\overline{\Omega})$ for almost all $T \in
[0,T_{loc}^{(\beta)}]$. Assume that $T_0$ is chosen such that
$h^{\frac{\alpha + 2}{2}}(\cdot,T_0) \in C^{1/2}(\overline{\Omega})$
and $h(x_0,T_0) = 0$ at some $x_0 \in \bar{\Omega}$. Then there
exists a constant $L$ such that
$$
h^{\frac{\alpha + 2}{2}}(x,T_0) = |h^{\frac{\alpha +
2}{2}}(x,T_0)-h^{\frac{\alpha + 2}{2}}(x_0,T_0)| \leq L |
x-x_0|^{1/2}.
$$
Hence $h(x,T_0) \leq L^{\frac{2}{\alpha + 2}} |
x-x_0|^{\frac{1}{\alpha + 2}}$, i.\,e. $h(.,T) \in
C^{\frac{1}{\alpha + 2}}(\bar{\Omega})$ for almost every $T \in
[0,T_{loc}^{(\beta)}]$.
\end{remark}

\begin{remark}[Rate of decrease]
For $a_1 \leq 0$ and $1/2 < n < 3$ we can generalize the results
from \cite[Theorem~1.1]{CarUl} in the following way:
$$
\int\limits_{\Omega} {h^{\alpha} h_{x}^2(x,t)\,dx} \leq C (1 +
t)^{- \frac{1}{2}} \text{ for }  \tfrac{n - 4}{2} \leq \alpha < 0,
$$
whence $ \| h - \bar{h} \|_{\infty} \leq C (1 + t)^{-
\frac{1}{4}}$ for any nonnegative strong solution $h$. Here $C = C
(a_1, \alpha, n, \bar{h}, \mathcal{E}^{(\alpha)}_{0}(0))$, and
$\bar{h} = \frac{1}{|\Omega|}\|h_0\|_1$. The proof is similar to
\cite{CarUl}.
\end{remark}

\subsection{Regularized Problem}\label{RegularizedProblem}

Given $\delta, \varepsilon > 0$, a regularized parabolic problem,
similar to that of Bernis and Friedman \cite{B8}, is considered:
\begin{numcases}
{(\textup{P}_{\delta,\varepsilon})}
 h_t  + \left({ f_{\delta\varepsilon}(h) (h_{xxx} + a_1 D''_{\varepsilon}(h)
 h_x)}\right)_x = 0, \qquad \hfill \label{D:1r'}\\
\tfrac{\partial^{i} h}{\partial x^i}(-a,t) = \tfrac{\partial^{i}
h}{\partial x^i}(a,t) \text{ for }
t > 0,\,i= \overline{0,3} , \hfill \label{D:2r'}\\
\qquad  \qquad h(x,0) = h_{0,\delta \varepsilon}(x), \hfill
\label{D:3r'}
\end{numcases}
where
\begin{equation}\label{D:reg1}
f_{\delta \varepsilon} (z): = f_{\varepsilon} (z) + \delta =
\tfrac{|z|^{s + n}}{|z|^s + \varepsilon |z|^n}+ \delta,\
D''_{\varepsilon}(z): = \tfrac{|z|^{m - n}}{1 + \varepsilon |z|^{m
- n}}
\end{equation}
$\forall\, z \in \mathbb{R}^1,\ \varepsilon >0,\ s \geqslant 4$.
The $\delta>0$ in (\ref{D:reg1}) makes the problem (\ref{D:1r'})
regular (i.e. uniformly parabolic). The parameter $\varepsilon$ is
an approximating parameter which has the effect of increasing the
degeneracy from $f(h) \sim |h|^n$ to $f_{\varepsilon}(h) \sim
h^s$. The nonnegative initial data, $h_0$, is approximated via
\begin{equation}\label{D:inreg}
\begin{gathered}
h_{0,\delta \varepsilon} \in C^{4+\gamma}(\Omega), \ h_{0,\delta
\varepsilon} \geqslant h_{0,\delta} + \varepsilon^\theta  \text{
for some } 0 < \theta < \tfrac{2}{2s -
3},\\
\tfrac{\partial^{i} h_{0,\delta \varepsilon}}{\partial x^i}(-a) =
\tfrac{\partial^{i}h_{0,\delta \varepsilon}}{\partial x^i}(a)
\text{ for } i=\overline{0,3}, \\
h_{0,\delta \varepsilon} \to h_{0}  \text{ strongly in }
H^1(\Omega) \text{ as } \delta, \varepsilon \to 0.
\end{gathered}
\end{equation}
 The $\varepsilon$ term in (\ref{D:inreg}) ``lifts'' the initial data so that it will be positive even if
$\delta = 0$ and the $\delta$ is involved in smoothing the initial
data from $H^1(\Omega)$ to $C^{4+\gamma}(\Omega)$.

\emph{Sketch of Proof}: By E\u{i}delman \cite[Theorem 6.3,
p.302]{Ed}, the regularized problem has the unique classical
solution $h_{\delta \varepsilon} \in C_{x,t}^{4+\gamma,1+\gamma/4}(
\Omega \times [0, \tau_{\delta \varepsilon}])$ for some time
$\tau_{\delta \varepsilon}
> 0$.
For any fixed values of  $\delta$ and $\varepsilon$, by E\u{i}delman
\cite [Theorem 9.3, p.316]{Ed} if one can prove a uniform in time an
a priori bound $|h_{\delta \varepsilon}(x,t)| \leq A_{\delta
\varepsilon}<\infty$ for some longer time interval $[0,T_{loc,\delta
\varepsilon}] \quad (T_{loc,\delta \varepsilon}
> \tau_{\delta \varepsilon}$) and for all $x \in \Omega$ then
Schauder-type interior estimates \cite [Corollary 2, p.213] {Ed}
imply that the solution $h_{\delta \varepsilon}$ can be continued
in time to be in $C_{x,t}^{4+\gamma,1+\gamma/4}( \Omega \times
[0,T_{loc,\delta \varepsilon}])$.

Although the solution $h_{\delta \varepsilon}$ is initially
positive, there is no guarantee that it will remain nonnegative.
The goal is to take $\delta \to 0$, $\varepsilon \to 0$ in such a
way that 1) $T_{loc,\delta \varepsilon} \to T_{loc} > 0$, 2) the
solutions $h_{\delta \varepsilon}$ converge to a (nonnegative)
limit, $h$, which is a generalized weak solution, and 3) $h$
inherits certain a priori bounds.  This is done by proving various
a priori estimates for $h_{\delta \varepsilon}$ that are uniform
in $\delta$ and $\varepsilon$ and hold on a time interval
$[0,T_{loc}]$ that is independent of $\delta$ and $\varepsilon$.
As a result, $\{ h_{\delta \varepsilon} \}$ will be a uniformly
bounded and equicontinuous (in the $C_{x,t}^{1/2,1/8}$ norm)
family of functions in $\bar{\Omega} \times [0, T_{loc}]$. Taking
$\delta \to 0$ will result in a family of functions $\{
h_{\varepsilon} \}$ that are classical, positive, unique solutions
to the regularized problem with $\delta = 0$. Taking $\varepsilon
\to 0$ will then result in the desired generalized weak solution
$h$. This last  step is where the possibility of non-unique weak
solutions arise; see \cite{B2} for simple examples of how such
constructions applied to $h_t = - (|h|^n h_{xxx})_x$ can result in
two different solutions arising from the same initial data.

\subsection{Dissipation of energy for positive solutions}\label{H}

\begin{figure}[ht] \label{Figure1}
\begin{center}
\includegraphics[height= 6 cm] {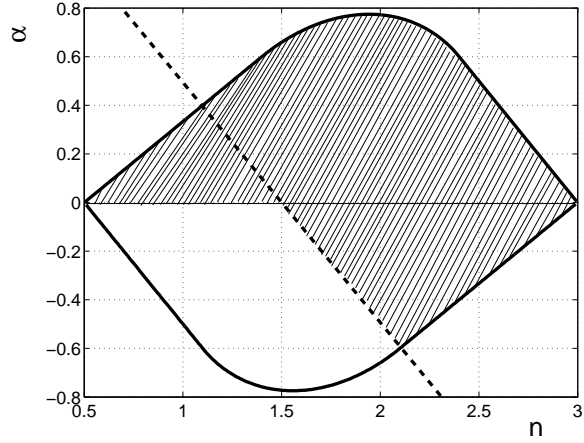}
\end{center}
\caption{ The dissipation region computed numerically by Matlab:
$\alpha$ versus $n$. The dashed line corresponds to $\alpha = 3/2 -
n$.}
\end{figure}

\begin{lemma}\label{H:lem1}
Let $\alpha $ belong to the full domain shown on Figure~1, and
\begin{equation} \label{Energy-new0}
\mathcal{E}_{\varepsilon}^{(\alpha)}(T) := \int\limits_{\Omega} {
\{\tfrac{1}{2} h^{\alpha} h_{x}^2(x,T) - a_1
\tilde{D}_{\varepsilon}(h(x,T)) \}\,dx},
\end{equation}
where $\tilde{D}''_{\varepsilon}(z):= z^{\alpha}
D''_{\varepsilon}(z)$. Then the unique positive classical solution $
h_{\varepsilon}$ of the problem ($P_{0,\varepsilon}$) satisfies
\begin{multline}\label{H:1}
\mathcal{E}^{(\alpha)}_{\varepsilon}(T) \leqslant
\mathcal{E}^{(\alpha)}_{\varepsilon}(0) + \mu \iint \limits_{Q_T}
{h^{\alpha - 2}f_{\varepsilon}(h)
D''_{\varepsilon}(h)h_x^4 dx dt} + \\
\varepsilon\,k_1 \int\limits_{\Omega} { h^{\alpha - s - 4}
f_{\varepsilon}^2(h) h_x^6 \,dx} + \varepsilon^2
k_2\int\limits_{\Omega} { h^{\alpha - 2s - 4} f_{\varepsilon}^3(h)
h_x^6 \,dx},
\end{multline}
where $k_i = k_i (\alpha, n,s)$ are constants, and $\mu = \mu
(\alpha, a_1)$ such that $\mu (0,a_1) = 0$ and $\mu (\alpha,a_1)
\leqslant 0$ for $a_1 \leqslant 0$.
\end{lemma}

Note that, although we use the same convenient notations introduced
in \cite{Laugesen}, the proof of Lemma~\ref{H:lem1} has essential
differences from the proof of Theorem~1 of \cite{Laugesen}. Indeed,
we introduce new ideas in order to estimate the lower-order term in
the equation (\ref{D:1r'}). In particular, the new quantity $N$ is
introduced, the quantity $R$ is modified, and so are the terms
involving the regularization parameter $\varepsilon$ in (\ref{H:8}).

\begin{proof}[Proof of Lemma~\ref{H:lem1}]
To prove the bound (\ref{H:1}), multiply (\ref{D:1r'}) with $\delta
= 0$ by $ - \frac{ \alpha}{2}  h^{\alpha - 1} h_x^2 - h^{\alpha}
h_{xx} - a_1 \tilde{D}'_{\varepsilon}(h) $, integrate over $\Omega$,
use integration by parts, apply the periodic boundary conditions
(\ref{D:2r'}), to find
\begin{multline}\label{H:2}
\tfrac{d}{dt}\mathcal{E}^{(\alpha)}_{\varepsilon}(t) = -
\int\limits_{\Omega} {h^{\alpha} f_{\varepsilon}(h) (h_{xxx} +
a_1 D''_{\varepsilon}(h)h_x)^2 \,dx} - \\
\tfrac{\alpha }{2} (\alpha - 1) \int\limits_{\Omega} { h^{\alpha -
2} h_x^3 f_{\varepsilon}(h) (h_{xxx} + a_1
D''_{\varepsilon}(h)h_x) \,dx} - \\
2 \alpha \int\limits_{\Omega} { h^{\alpha - 1} h_x h_{xx}
f_{\varepsilon}(h) (h_{xxx} + a_1 D''_{\varepsilon}(h)h_x) \,dx}.
\end{multline}
The equality (\ref{H:2})  can be rewritten as
\begin{equation}\label{H:3}
\tfrac{d}{dt}\mathcal{E}^{(\alpha)}_{\varepsilon}(t) = - R^2 -
2\alpha \,RS - \tfrac{\alpha }{2}(\alpha - 1)\,RL,
\end{equation}
where the quantities
$$
R := \langle (h^{\alpha}f_{\varepsilon}(h))^{1/2} (h_{xxx} + a_1
D''_{\varepsilon}(h)h_x) | = |
(h^{\alpha}f_{\varepsilon}(h))^{1/2} (h_{xxx} +  a_1
D''_{\varepsilon}(h)h_x) \rangle ,
$$
$$
S := \langle (h^{\alpha - 2}f_{\varepsilon}(h))^{1/2} h_x h_{xx} | =
| (h^{\alpha - 2}f_{\varepsilon}(h))^{1/2} h_x h_{xx} \rangle ,
$$
$$
L := \langle (h^{\alpha - 4}f_{\varepsilon}(h))^{1/2} h_x^3 | = |
(h^{\alpha - 4}f_{\varepsilon}(h))^{1/2} h_x^3 \rangle,
$$
$$
N := \langle (h^{\alpha - 2}f_{\varepsilon}(h)
D''_{\varepsilon}(h))^{1/2} h_x^2 | = | (h^{\alpha -
2}f_{\varepsilon}(h) D''_{\varepsilon}(h))^{1/2} h_x^2 \rangle,
$$
each represent half of an inner product in $L^2 (\Omega)$. We will
need the following integration by parts formulas
\begin{multline}\label{H:4}
SL = - \tfrac{1}{5}(\alpha - 3) \int\limits_{\Omega} { h^{\alpha -
4} f_{\varepsilon}(h) h_x^6 \,dx} - \tfrac{1}{5}
\int\limits_{\Omega} { h^{\alpha - 3} f'_{\varepsilon}(h) h_x^6
\,dx} = \\
 - \tfrac{1}{5}(\alpha + n - 3) L^2 - \tfrac{1}{5}\varepsilon(s -n)
\int\limits_{\Omega} { h^{\alpha - s - 4} f_{\varepsilon}^2(h) h_x^6
\,dx},
\end{multline}
\begin{multline}\label{H:5}
RL = -(\alpha + n - 2 ) SL - \varepsilon(s -n) \int\limits_{\Omega}
{ h^{\alpha - s - 3} f_{\varepsilon}^2(h)
h_x^4 h_{xx}\,dx} - 3 S^2 + a_1 N^2 = \\
\tfrac{1}{5}(\alpha + n - 2 )(\alpha + n - 3) L^2 - 3 S^2 + a_1
N^2 - \varepsilon(s -n) \int\limits_{\Omega} { h^{\alpha - s - 3}
f_{\varepsilon}^2(h)
h_x^4 h_{xx}\,dx} + \\
\tfrac{1}{5}\varepsilon(s -n)(\alpha + n - 2 )\int\limits_{\Omega} {
h^{\alpha - s - 4} f_{\varepsilon}^2(h) h_x^6 \,dx} =
\tfrac{1}{5}(\alpha + n - 2 )(\alpha + n - 3) L^2
- 3 S^2 + \\
a_1 N^2  + \tfrac{1}{5}\varepsilon(s -n)(2\alpha + 3n - s -
5)\int\limits_{\Omega} { h^{\alpha - s - 4} f_{\varepsilon}^2(h)
h_x^6 \,dx} + \\
\tfrac{2}{5}\varepsilon^2(s -n)^2 \int\limits_{\Omega} { h^{\alpha -
2s - 4} f_{\varepsilon}^3(h) h_x^6 \,dx}.
\end{multline}
Here we use the auxiliary equality $ f'_{\varepsilon}(z) = n z^{-1}
f_{\varepsilon}(z) + \varepsilon (s -
n)z^{-(s+1)}f_{\varepsilon}^2(z)$. Thus, from (\ref{H:3}) we have
\begin{multline}\label{H:6}
\tfrac{d}{dt}\mathcal{E}^{(\alpha)}_{\varepsilon}(t) +
\tfrac{\varepsilon^2}{5}\alpha (\alpha - 1)(s -n)^2
\int\limits_{\Omega} { h^{\alpha - 2s - 4} f_{\varepsilon}^3(h)
h_x^6 \,dx} + \tfrac{a_1}{2}\alpha (\alpha - 1)N^2 = \\
- R^2 - 2\alpha \,RS - \tfrac{\alpha }{10}(\alpha - 1)(\alpha + n
- 2 )(\alpha + n - 3) L^2 + \tfrac{3\alpha}{2}
 (\alpha - 1)S^2  + \\
\tfrac{\varepsilon}{10}\alpha (\alpha - 1)(s -n)(s - 2\alpha - 3n
+ 5)\int\limits_{\Omega} { h^{\alpha - s - 4} f_{\varepsilon}^2(h)
h_x^6 \,dx}.
\end{multline}
Our next step is to express (\ref{H:6}) as the negative of a sum of
squares to obtain the energy dissipation. To achieve this, we use
(\ref{H:4}) and (\ref{H:5}) to deduce that for all $\kappa \in
\mathbb{R}^1$,
\begin{multline}\label{H:8}
\tfrac{d}{dt}\mathcal{E}^{(\alpha)}_{\varepsilon}(t) = - (R +
\alpha \,S + \kappa \, L)^2 + \beta ( S+ \tfrac{1}{5}
(\alpha + n - 3) \, L )^2  + \gamma \, L^2 + \mu\,N^2 +\\
\varepsilon\,k_1 \int\limits_{\Omega} { h^{\alpha - s - 4}
f_{\varepsilon}^2(h) h_x^6 \,dx} + \varepsilon^2
k_2\int\limits_{\Omega} { h^{\alpha - 2s - 4} f_{\varepsilon}^3(h)
h_x^6 \,dx} ,
\end{multline}
where
\begin{multline*}
k_1 = \tfrac{2}{25}(s - n) \biggl(5 \kappa (\alpha - s + 3n -5) +
\tfrac{5\alpha}{4} (\alpha - 1)(s - 2\alpha - 3n + 5) + \\
(\alpha + n -3)( \tfrac{\alpha}{2}  (5\alpha - 3) - 6
\kappa)\biggr),\ k_2 =  \tfrac{1}{5}(s - n)^2 \biggl( 4\kappa -
\alpha (\alpha - 1) \biggr),
\end{multline*}
$$
\beta =  \tfrac{\alpha}{2}  (5\alpha - 3) - 6 \kappa , \, \ \mu =
a_1 \biggl(2\kappa - \tfrac{\alpha }{2} (\alpha - 1) \biggr),
$$
\begin{multline*}
\gamma =  \kappa^2 - \tfrac{2}{25} \kappa (\alpha + n - 3)
\bigl( 5(2-n) + 3(\alpha + n - 3) \bigr) - \\
\tfrac{\alpha}{50}  (\alpha + n - 3) \bigl( 5(\alpha - 1)(\alpha
+ n - 2) - (5 \alpha - 3)(\alpha + n - 3) \bigr)  = \\
\kappa^2 - \tfrac{6}{25} \kappa (\alpha + n - 3) \bigl( \alpha -
\tfrac{2n-1}{3} \bigr) - \tfrac{3}{50} \alpha (\alpha + n - 3)
\bigl( \alpha - \tfrac{2n-1}{3} \bigr) .
\end{multline*}
Now we have to choose the parameter $\kappa$ in such a way that
$\beta \leqslant 0$ and $\gamma \leqslant 0$. In this case, the
parameter $\mu > 0$ for $a_1 > 0$, and $\mu \leqslant 0$  for $a_1
\leqslant 0$. According to \cite{Laugesen}, we can find $\kappa$
such that $\beta \leqslant 0$, and $\gamma \leqslant 0$ when $1/2 <
n < 3$, see also Figure~1 where this region was computed numerically
by Matlab (see \cite{Laugesen} for the explicit form of the domain).

\subsection{Limit process in (\ref{H:1})}

Rewrite the integral $\iint\limits_{Q_T} {\varepsilon
h_{\varepsilon}^{\alpha - s - 4}
f_{\varepsilon}^2(h_{\varepsilon}) h_{\varepsilon,x}^6 \,dx dt}$
in the form
$$
\iint\limits_{Q_T} {\varepsilon h_{\varepsilon}^{\alpha - s - 4}
f_{\varepsilon}^2(h_{\varepsilon}) h_{\varepsilon,x}^6 \,dx dt} =
\iint\limits_{Q_T} { \tfrac{\varepsilon \,h_{\varepsilon}^{\alpha
+ s - n}}{(h_{\varepsilon}^{s -n} + \varepsilon)^2}
h_{\varepsilon}^{n - 4} h_{\varepsilon,x}^6 \,dx dt}.
$$
Using the Young's inequality
\begin{equation}\label{H:Young}
a b \leqslant \tfrac{a^p}{p} + \tfrac{b^q}{q}  \Rightarrow  p \,a
b \leqslant  a^p + (p - 1) b^q, \ \tfrac{1}{p} + \tfrac{1}{q} =1,
\end{equation}
with $a = z^{\frac{s - n}{p}}$ and $b = \bigl(
\frac{\varepsilon}{p - 1} \bigr)^{\frac{1}{q}}$, we deduce
$$
\varepsilon \tfrac{z^{\alpha + s - n}}{(z^{s - n} +
\varepsilon)^2} \leqslant \varepsilon \tfrac{z^{\alpha}}{z^{s - n}
+ \varepsilon} \leqslant \tfrac{(p-1)^{\frac{1}{q}}}{p} \tfrac{
\varepsilon \, z^{\alpha}}{z^{\frac{s - n}{p}}
\varepsilon^{\frac{1}{q}}} = \tfrac{(p-1)^{\frac{1}{q}}}{p}
\varepsilon^{\frac{q - 1}{q}} z^{\frac{p\alpha - s + n }{p}},
$$
choosing $p = \frac{s - n}{\alpha} > 1$ and $q = \frac{s - n}{s -n
- \alpha}  > 1$ ($\Rightarrow 0 < \alpha < s - n$), we find
$$
\varepsilon \tfrac{z^{\alpha + s - n}}{(z^{s - n} +
\varepsilon)^2} \leqslant \tfrac{\alpha}{s - n} \bigl( \tfrac{ s -
n - \alpha}{\alpha} \bigr)^{\frac{s - n - \alpha}{s - n}}
\varepsilon^{\frac{\alpha}{s - n}}.
$$
Similarly, we deal with the integral $\varepsilon^2
\int\limits_{\Omega} { h_{\varepsilon}^{\alpha - 2s - 4}
f_{\varepsilon}^3(h) h_{\varepsilon,x}^6 \,dx}$. Due to
Lemma~\ref{A.3} and (\ref{C:ddd2'}), $ \iint\limits_{Q_T}
{h_{\varepsilon}^{n - 4}h_{\varepsilon,x}^6 \,dx dt} $ is
uniformly bounded then
\begin{multline}\label{H:lp1}
\Bigl | \iint\limits_{Q_T} { ( k_1 \varepsilon
h_{\varepsilon}^{\alpha - s - 4}
f_{\varepsilon}^2(h_{\varepsilon}) + k_2 \varepsilon^2
h_{\varepsilon}^{\alpha - 2s - 4}
f_{\varepsilon}^3(h_{\varepsilon})) h_{\varepsilon,x}^6 \,dx dt}
\Bigr | \leqslant \\
C\,\varepsilon^{\frac{\alpha}{s - n}} \iint\limits_{Q_T}
{h_{\varepsilon}^{n - 4}h_{\varepsilon,x}^6 \,dx dt} \leqslant
C\,\varepsilon^{\frac{\alpha}{s - n}},
\end{multline}
where the positive constant $C$ is independent of $\varepsilon$.
Letting $\varepsilon \to 0$, from (\ref{H:lp1}) we obtain
\begin{equation}\label{H:lp2}
( k_1 \varepsilon h_{\varepsilon}^{\alpha - s - 4}
f_{\varepsilon}^2(h_{\varepsilon}) + k_2 \varepsilon^2
h_{\varepsilon}^{\alpha - 2s - 4}
f_{\varepsilon}^3(h_{\varepsilon})) h_{\varepsilon,x}^6 \to 0
\text{ in } L^1(Q_T)
\end{equation}
for $0 < \alpha < s - n$.

Now, we show (\ref{H:lp2}) for the case of $\alpha < 0$. Rewrite
the integral $\iint\limits_{Q_T} {\varepsilon
h_{\varepsilon}^{\alpha - s - 4}
f_{\varepsilon}^2(h_{\varepsilon}) h_{\varepsilon,x}^6 \,dx dt}$
in the form
$$
\iint\limits_{Q_T} {\varepsilon h_{\varepsilon}^{\alpha - s - 4}
f_{\varepsilon}^2(h_{\varepsilon}) h_{\varepsilon,x}^6 \,dx dt} =
\iint\limits_{Q_T} { \tfrac{\varepsilon \,h_{\varepsilon}^{s -
2}}{(h_{\varepsilon}^{s -n} + \varepsilon)^2}
h_{\varepsilon}^{\alpha - 2} h_{\varepsilon,x}^6 \,dx dt}.
$$
Using the inequality (\ref{H:Young}) with $a = z^{\frac{s -
n}{p}}$ and $b = \bigl( \frac{\varepsilon}{p - 1}
\bigr)^{\frac{1}{q}}$, we obtain
$$
\varepsilon \tfrac{z^{s - 2}}{(z^{s - n} + \varepsilon)^2}
\leqslant \tfrac{(p-1)^{\frac{2}{q}}}{p^2} \tfrac{ \varepsilon \,
z^{s-2}}{z^{\frac{2(s - n)}{p}} \varepsilon^{\frac{2}{q}}} =
\tfrac{(p-1)^{\frac{2}{q}}}{p^2} \varepsilon^{\frac{q - 2}{q}}
z^{\frac{p(s-2) - 2(s - n)}{p}},
$$
choosing $p = \frac{2(s - n)}{s - 2 + \alpha - \beta} < 2$ and $q
= \frac{2(s - n)}{2(s -n) - s + 2 - \alpha + \beta}  > 2$
($\Rightarrow n > 2 - \alpha + \beta$), we find
$$
\varepsilon \tfrac{z^{s - 2}}{(z^{s - n} + \varepsilon)^2}
\leqslant \bigl( \tfrac{s - 2 + \alpha - \beta}{2(s - n)}\bigr)^2
\bigl( \tfrac{ s - 2(n -1) - \alpha + \beta}{s -2 + \alpha -
\beta} \bigr)^{\frac{s - 2(n - 1)- \alpha + \beta}{s - n}}
\varepsilon^{\frac{n-2 + \alpha -\beta}{s - n}} z^{\beta -
\alpha},
$$
where $\beta \in (-1/2,1)$ follows from (\ref{Linf_H2-2}).
Similarly, we deal with the integral $\varepsilon^2
\int\limits_{\Omega} { h_{\varepsilon}^{\alpha - 2s - 4}
f_{\varepsilon}^3(h) h_{\varepsilon,x}^6 \,dx}$. Due to $h \in
L^{\infty}(0,T; H^1(\Omega))$ and (\ref{Linf_H2-2}), $
\iint\limits_{Q_T} {h_{\varepsilon}^{\alpha - 2}h_{\varepsilon,x}^6
\,dx dt} $ is uniformly bounded then
\begin{multline}\label{H:lp1-00}
\Bigl | \iint\limits_{Q_T} { ( k_1 \varepsilon
h_{\varepsilon}^{\alpha - s - 4}
f_{\varepsilon}^2(h_{\varepsilon}) + k_2 \varepsilon^2
h_{\varepsilon}^{\alpha - 2s - 4}
f_{\varepsilon}^3(h_{\varepsilon})) h_{\varepsilon,x}^6 \,dx dt}
\Bigr | \leqslant \\
C\,\varepsilon^{\frac{n-2 + \alpha -\beta}{s - n}}
\iint\limits_{Q_T} {h_{\varepsilon}^{\beta - 2}
h_{\varepsilon,x}^6 \,dx dt} \leqslant C\,\varepsilon^{\frac{n-2 +
\alpha -\beta}{s - n}}\iint\limits_{Q_T} {h_{\varepsilon}^{\beta -
2} h_{\varepsilon,x}^4 \,dx dt}
 \leqslant \\
 C\,\varepsilon^{\frac{n-2 + \alpha -\beta}{s - n}},
\end{multline}
where the positive constant $C$ is independent of $\varepsilon$.
Letting $\varepsilon \to 0$, we obtain (\ref{H:lp2}) for
$\frac{3}{2} - n < \alpha < 0$ and $\frac{3}{2} < n < 3$. In view
of the Lebesgue's theorem, we have
\begin{equation}\label{H:l3}
\iint \limits_{Q_T} {h_{\varepsilon}^{\alpha -
2}f_{\varepsilon}(h_{\varepsilon})
D''_{\varepsilon}(h_{\varepsilon}) h_{\varepsilon,x}^4 \, dx dt}
\to \iint \limits_{Q_T} {h^{\alpha + m - 2}h_x^4 \,dx dt}
\end{equation}
if $m > 0$ and $\alpha > - \frac{1}{2} - m$, due to $h \in
L^{\infty}(0,T; H^1(\Omega))$ and (\ref{Linf_H2-2}).

Integrating (\ref{H:8}) over the time interval, and letting
$\varepsilon \to 0$, in view of (\ref{H:lp2}) and (\ref{H:l3}), we
obtain (\ref{H:1}) for some subinterval $I$ for $0 \leq \alpha < 1$
and $\frac{1}{2} < n < 3$ or for $- 1 < \alpha < 0$ and $\frac{3}{2}
< n < 3$. Note that, the convergence on the left-hand side follows
from Fatou's lemma and from the corresponding a priori estimate
(see, for example, \cite{B8,BertPugh1996,BP2,T4}).
\end{proof}

\subsection{Proof of Theorem~\ref{D:Th1}}\label{E}

Taking the limit $\varepsilon \to 0$ we obtain
\begin{equation}\label{E:1}
\mathcal{E}^{(\alpha)}_{0}(T) + \gamma \iint \limits_{Q_T}
{h^{\alpha + n  - 4}h_x^6 dx dt} \leqslant
\mathcal{E}^{(\alpha)}_{0}(0) + \mu \iint \limits_{Q_T} {h^{\alpha
+ m - 2}h_x^4 dx dt}.
\end{equation}
Now, we estimate  $\iint \limits_{Q_T} {h^{\alpha + m - 2}h_x^4 dx
dt}$. Using the H\"{o}lder inequality, we obtain
\begin{equation}\label{E:2}
\iint \limits_{Q_T} {h^{\alpha + m - 2}h_x^4 dx dt} \leqslant \int
\limits_{0}^T {\Bigl( \int \limits_{\Omega} {h^{\alpha + n -
4}h_x^6 dx} \Bigr)^{\frac{2}{3}} \Bigl( \int \limits_{\Omega}
{h^{\alpha + 3m - 2n + 2}dx} \Bigr)^{\frac{1}{3}} dt}.
\end{equation}
Applying Lemma~\ref{A.4} to $v = h^{\frac{\alpha + n + 2}{6}}$
with $a = \frac{6(\alpha + 3m - 2n + 2)}{\alpha + n + 2} $, $d =
6$, $b = \frac{6}{\alpha + n + 2} < a \ ( \Rightarrow \alpha > 2n
- 3m - 1)$, $i =0$, and $j = 1$, we deduce
\begin{multline}\label{E:3}
\int \limits_{\Omega} {h^{\alpha + 3m - 2n + 2} dx } \leqslant d_1
\Bigl( \int \limits_{\Omega} {v_x^6 dx} \Bigr)^{\frac{\alpha + 3m
- 2n + 1}{\alpha + n + 7}} \Bigl( \int \limits_{\Omega} {h\,dx}
\Bigr)^{\frac{3(2\alpha + 5 m - 3n + 4)}{\alpha + n + 7}} + \\
d_2 \Bigl( \int \limits_{\Omega} {h\,dx} \Bigr)^{\alpha + 3m - 2n
+ 2} \leqslant c_1 M^{\frac{3(2\alpha + 5 m - 3n + 4)}{\alpha + n
+ 7}}  \Bigl( \int \limits_{\Omega} {h^{\alpha + n - 4}h_x^6  dx}
\Bigr)^{\frac{\alpha + 3m - 2n + 1}{\alpha + n + 7}} + \\
c_2 M^{\alpha + 3m - 2n + 2}.
\end{multline}
Substituting (\ref{E:3}) in (\ref{E:2}), we find
\begin{multline}\label{E:4}
\iint \limits_{Q_T} {h^{\alpha + m - 2}h_x^4 dx dt} \leqslant c_1
M^{\frac{2\alpha + 5 m - 3n + 4}{\alpha + n + 7}} \int
\limits_{0}^T {\Bigl( \int \limits_{\Omega} {h^{\alpha + n -
4}h_x^6 dx} \Bigr)^{\frac{\alpha + m + 5}{\alpha + n + 7}} dt} +
\\
c_2 M^{\frac{\alpha + 3m - 2n + 2}{3}} \int \limits_{0}^T {\Bigl(
\int \limits_{\Omega} {h^{\alpha + n - 4}h_x^6 dx}
\Bigr)^{\frac{2}{3}} dt}.
\end{multline}

If $m < n + 2$ then, using Young's inequality, from (\ref{E:4}) we
arrive at
\begin{multline}\label{E:5}
\iint \limits_{Q_T} {h^{\alpha + m - 2}h_x^4 dx dt} \leqslant
\epsilon \iint \limits_{Q_T} {h^{\alpha + n  - 4}h_x^6 dx dt} +\\
C(\epsilon) T(c_1 M^{\frac{2\alpha + 5 m - 3n + 4}{n + 2 - m}} +
c_2 M^{\alpha + 3m - 2n + 2}) .
\end{multline}
Substituting (\ref{E:5}) in (\ref{E:1}), and choosing $\epsilon$
small enough, we obtain
\begin{equation}\label{E:1-2}
\mathcal{E}^{(\alpha)}_{0}(T) \leqslant
\mathcal{E}^{(\alpha)}_{0}(0) +  T(C_1 M^{\frac{2\alpha + 5 m - 3n
+ 4}{n + 2 - m}} + C_2 M^{\alpha + 3m - 2n + 2})
\end{equation}
for $\alpha > 2n - 3m - 1$ and $m < n + 2$. Here $C_2 = 0$ if
$\Omega$ is unbounded or $h$ is compactly supported. In particular,
if $0 < \alpha + 3m - 2n + 2 \leqslant 1$, i.\,e. $2n - 3m - 2 <
\alpha \leqslant 2n - 3m - 1$ then, using the H\"{o}lder inequality
and applying Young's inequality, from (\ref{E:2}) we obtain
\begin{multline}\label{E:2-2}
\iint \limits_{Q_T} {h^{\alpha + m - 2}h_x^4 dx dt} \leqslant
\epsilon \iint \limits_{Q_T} {h^{\alpha + n  - 4}h_x^6 dx dt} +\\
C(\epsilon)|\Omega|^{2n - 3m - 1 - \alpha} T\, M^{\alpha + 3 m -
2n + 2}
\end{multline}
for $2n - 3m - 2 < \alpha \leqslant 2n - 3m - 1$. Substituting
(\ref{E:2-2}) in (\ref{E:1}), and choosing $\epsilon$ small enough,
we obtain
\begin{equation}\label{E:1-4}
\mathcal{E}^{(\alpha)}_{0}(T) \leqslant
\mathcal{E}^{(\alpha)}_{0}(0) +  C_3 T\, M^{\alpha + 3m - 2n + 2}.
\end{equation}

If $m = n + 2$ then, using Young's inequality, from (\ref{E:4}) we
deduce
\begin{multline}\label{E:6}
\iint \limits_{Q_T} {h^{\alpha + m - 2}h_x^4 dx dt} \leqslant c_1
M^{2} \iint \limits_{Q_T} {h^{\alpha + n - 4}h_x^6 dx dt } +
\\
\epsilon \iint \limits_{Q_T} {h^{\alpha + n - 4}h_x^6 dx dt } +
C(\epsilon)T\, M^{\alpha + n + 8 } .
\end{multline}
Substituting (\ref{E:6}) in (\ref{E:1}), and choosing $\epsilon$
enough small, we obtain
\begin{equation}\label{E:1-3}
\mathcal{E}^{(\alpha)}_{0}(T) \leqslant
\mathcal{E}^{(\alpha)}_{0}(0) +   C_2 T\, M^{\alpha + n + 8}
\end{equation}
for $\alpha > - n -7$, $m = n + 2$ and $M \leqslant M_c$. Here $C_2
= 0$ if $\Omega$ is unbounded or $h$ is compactly supported.

{\bf Acknowledgement.} The authors thank R.~S.~Laugesen and
A.~Burchard for useful comments and discussions.  The research of
M.~Chu\-gu\-nova is supported by the NSERC Postdoctoral Fellowship.
R.~M.~Ta\-ra\-nets would like to thank M.~C. Pugh for the
hospitality of the University of Toronto.

\appendix

\section*{Appendix A}

\renewcommand{\thesection}{A}\setcounter{lemma}{0}

\begin{lemma}\label{A.3}
(\cite{G4,BSS}) Let $\Omega \subset \mathbb{R}^N,\ N < 6$, be a
bounded convex domain with smooth boundary, and let $n \in \bigl(2 -
\sqrt{1 -\tfrac{N}{N + 8}}, 3 \bigr)$ for $N > 1$, and $\frac{1}{2}
< n < 3$ for $N = 1$. Then the following estimate holds for any
positive functions $v \in H^2(\Omega)$ such that $\nabla v \cdot
\vec{n} = 0$ on $\partial\Omega$ and $\int\limits_{\Omega}{v^n
|\nabla\Delta v|^2} < \infty$:
\begin{multline*}
\int\limits_{\Omega} {\varphi ^6 \{ v^{n - 4}|\nabla v|^6 + v^{n -
2} |D^2 v|^2 |\nabla v|^2 \}} \leqslant \\
c
\Bigl\{\int\limits_{\Omega} {\varphi ^6 v^n |\nabla \Delta v|^2} +
\int\limits_{\{ \varphi > 0\} } {v^{n + 2} |\nabla \varphi|^6}
\Bigr\},
\end{multline*}
where $\varphi \in C^2 (\Omega)$ is an arbitrary nonnegative
function such that the tangential component of $\nabla \varphi$ is
equal to zero on $\partial\Omega$, and the constant $c > 0$ is
independent of $v$.
\end{lemma}

\begin{lemma}\label{A.4}
(\cite{N1})  If $\Omega  \subset \mathbb{R}^N $ is a bounded
domain with piecewise-smooth boundary, $a > 1$, $b \in (0, a),\ d
> 1,$ and $0 \leqslant i < j,\ i,j \in \mathbb{N}$, then there
exist positive constants $d_1$ and $d_2$ $(d_2 = 0 \text{ if }
\Omega$ is unbounded$)$ depending only on $\Omega ,\ d,\ j,\ b,$
and $N$ such that the following inequality is valid for every
$v(x) \in W^{j,d} (\Omega ) \cap L^b (\Omega )$:
$$
\left\| {D^i v} \right\|_{L^a (\Omega )}  \leqslant d_1 \left\|
{D^j v} \right\|_{L^d (\Omega )}^\theta  \left\| v \right\|_{L^b
(\Omega )}^{1 - \theta }  + d_2 \left\| v \right\|_{L^b (\Omega )}
,\ \theta  = \frac{{\tfrac{1} {b} + \tfrac{i} {N} - \tfrac{1}
{a}}} {{\tfrac{1} {b} + \tfrac{j} {N} - \tfrac{1} {d}}} \in \left[
{\tfrac{i} {j},1} \right)\!\!.
$$
\end{lemma}

\def\cprime{$'$} \def\cprime{$'$} \def\cprime{$'$} \def\cprime{$'$}


\begin{thebibliography}{10}

\bibitem{B2}
Elena Beretta, Michiel Bertsch, and Roberta Dal~Passo.
\newblock Nonnegative solutions of a fourth-order nonlinear degenerate
  parabolic equation.
\newblock {\em Arch. Rational Mech. Anal.}, 129(2):175--200, 1995.

\bibitem{BSS}
Francisco Bernis.
\newblock Finite speed of propagation for thin viscous flows when {$2\leq
  n<3$}.
\newblock {\em C. R. Acad. Sci. Paris S\'er. I Math.}, 322(12):1169--1174,
  1996.

\bibitem{B8}
Francisco Bernis and Avner Friedman.
\newblock Higher order nonlinear degenerate parabolic equations.
\newblock {\em J. Differential Equations}, 83(1):179--206, 1990.

\bibitem{Bernoff}
Andrew~J. Bernoff and Andrea~L. Bertozzi.
\newblock Singularities in a modified {K}uramoto-{S}ivashinsky equation
  describing interface motion for phase transition.
\newblock {\em Phys. D}, 85(3):375--404, 1995.

\bibitem{BertPugh1996}
A.~L. Bertozzi and M.~Pugh.
\newblock The lubrication approximation for thin viscous films: regularity
and long-time behavior of weak solutions.
\newblock {\em Comm. Pure Appl. Math.}, 49(2):85--123, 1996.

\bibitem{BP1}
A.~L. Bertozzi and M.~C. Pugh.
\newblock Long-wave instabilities and saturation in thin film equations.
\newblock {\em Comm. Pure Appl. Math.}, 51(6):625--661, 1998.

\bibitem{BP2}
A.~L. Bertozzi and M.~C. Pugh.
\newblock Finite-time blow-up of solutions of some long-wave unstable thin
film equations.
\newblock {\em Indiana Univ. Math. J.}, 49(4):1323--1366, 2000.

\bibitem{BertozziBrenner}
Andrea~L. Bertozzi, Michael~P. Brenner, Todd~F. Dupont, and Leo~P.
Kadanoff.
\newblock Singularities and similarities in interface flows.
\newblock In {\em Trends and perspectives in applied mathematics}, volume 100
  of {\em Appl. Math. Sci.}, pages 155--208. Springer, New York, 1994.

\bibitem{CarUl}
E.~Carlen, S.~Ulusoy.
\newblock{An entropy dissipation-entropy estimate for a thin film type equation. }
\newblock{\em Comm. Math. Sci.}, {3}({2}):{171--178}, 2005.

\bibitem{Constantin}
P.~Constantin, T.~F.~Dupont, R.~E.~Goldstein, Leo P.~Kadanoff,
M.~J.~Shelley, and S.~M.~Zhou.
\newblock {Droplet breakup in a model of the Hele-Shaw cell}.
\newblock {\em {Physical review E}}, {47}({6}):{4169--4181}, {june} {1993}.

\bibitem{Enrhard}
P.~Ehrhard.
\newblock {The spreading of hanging drops}.
\newblock {\em {Journal of colloid and interface science}},
  {168}({1}):{242--246}, {nov} {1994}.

\bibitem{Ed}
S.~D. {\`E}{\u\i}del{\cprime}man.
\newblock {\em Parabolic systems}.
\newblock Translated from the Russian by Scripta Technica, London.
  North-Holland Publishing Co., Amsterdam, 1969.


\bibitem{G4}
G{\"u}nther Gr{\"u}n.
\newblock Droplet spreading under weak slippage: a basic result on finite
speed
  of propagation.
\newblock {\em SIAM J. Math. Anal.}, 34(4):992--1006 (electronic), 2003.

\bibitem{JM}
Ansgar~Jungel, and Danial~Matthes.
\newblock An algorithmic construction of entropies in higher-order
nonlinear PDEs.
\newblock {\em Nonlinearity}, 19:633--659, 2006.

\bibitem{Laugesen}
R.~S. Laugesen. \newblock New dissipated energies for the thin
fluid film equation.
\newblock {\em Commun. Pure Appl. Anal.},  4 (3): 613--634, 2005.

\bibitem{N1}
L.~Nirenberg.
\newblock An extended interpolation inequality.
\newblock {\em Ann. Scuola Norm. Sup. Pisa (3)}, 20:733--737, 1966.

\bibitem{T4}
A.~E. Shishkov and R.~M. Taranets.
\newblock On the equation of the flow of thin films with nonlinear convection
  in multidimensional domains.
\newblock {\em Ukr. Mat. Visn.}, 1(3):402--444, 447, 2004.



\bibitem{Tudorascu}
A.~Tudorascu.
\newblock Lubrication approximation for thin viscous films:
 asymptotic behavior of nonnegative solutions.
\newblock {\em Communications in PDE}, 32:1147--1172, 2007.

\bibitem{wit1}
Thomas~P. Witelski and Andrew~J. Bernoff.
\newblock Stability of self-similar solutions for van der {W}aals driven thin
  film rupture.
\newblock {\em Phys. Fluids}, 11(9):2443--2445, 1999.

\end{thebibliography}
\end{document}